\documentclass[a4paper,10pt]{amsart}
\usepackage[english]{babel}
\usepackage[latin1]{inputenc}
\usepackage{amsmath,amsfonts,amssymb,amsthm,amscd,array,stmaryrd,mathrsfs,bbm}
\usepackage[all]{xy}
\usepackage{anysize}\marginsize{25mm}{25mm}{30mm}{30mm}
\allowdisplaybreaks[4]
\usepackage[misc]{ifsym}
\usepackage{textcomp,booktabs}
\usepackage{graphicx}
\usepackage[usenames,dvipsnames]{color}
\usepackage{colortbl}  
\definecolor{mygray}{gray}{.7}
\usepackage[colorlinks,unicode=false,linkcolor=blue,citecolor=blue]{hyperref}

\numberwithin{equation}{section}
\theoremstyle{plain}
\newtheorem{theorem}{Theorem}

\newtheorem{lemma}{Lemma}[section]
\newtheorem{corollary}[lemma]{Corollary}

\theoremstyle{definition}
\newtheorem{definition}[lemma]{Definition}
\newtheorem{remark}[lemma]{Remark}
\newtheorem{example}[lemma]{Example}

\usepackage{tikz-cd}

\title{Sweedler duality for Hom-(co)algebras and Hom-(co)modules}

\date{\today}



\author{Jiacheng Sun}
\address{School of Mathematics, Southeast University, Nanjing 211189, China}
\email{220242018@seu.edu.cn}
	
\author{Shuanhong Wang}
\address{Shing-Tung Yau Center, School of Mathematics, Southeast University, Nanjing 210096, China}
\email{shuanhwang@seu.edu.cn}

\author{Chi Zhang}
\address{College of Sciences, Northeastern University, Shenyang, Liaoning 110004, China}
\email{zhangchi@mail.neu.edu.cn}
\thanks{Corresponding author: Chi Zhang}
	
\author{Haoran Zhu}
\address{School of Physical and Mathematical Sciences, Nanyang Technological University, 21 Nanyang Link, 637371, Singapore}
\email{haoran.zhu@ntu.edu.sg}
	
\begin{document}
	
	\begin{abstract}
         We establish a dual version of infinite-dimensional Hom-algebras and Hom-modules by using the Sweedler duality construction.  Additionally, linear morphisms between infinite-dimensional Hom-algebras (resp. Hom-modules) and Hom-coalgebras (resp. Hom-comodules) are derived under this construction. As an application, we present a Hom-type binary linearly recursive sequences and show that the Sweedler duality construction can be utilized to determine the minimal polynomials of finite-codimensional ideals.
	\end{abstract}

        \maketitle

        \section{Introduction}
	Algebraic deformation theory was first studied systematically and in depth by Gerstenhaber \cite{G} and is now one of the most important branches of algebra. In recent years, the introduction
    of another type of algebraic deformation algebra, namely Hom-algebra, has attracted the attention of many algebraists. In 2006, Hartwig, Larsson, and Silvestrov \cite{HLS} introduced Hom-Lie
    algebras to describe $q$-deformations of Witt and Virasoro algebras. 
        The quasi-Lie and quasi-Hom-Lie algebras were introduced by Larsson and Silvestrov~\cite{LS} in 2005. Later, in 2008, Makhlouf and Silvestrov~\cite{MS1} studied Hom-associative, Hom-Leibniz, Hom-Lie, Lie-admissible, and flexible Hom-algebras. Building on this foundation, the authors \cite{MS2,MS3} generalized the general (co)algebraic structure to Hom-(co)algebras, which essentially extends the concept of associative algebras by deforming the associativity law into a Hom-associativity condition. Later, Yau \cite{Y1,Y2,Y3,Y4} proposed the concepts of Hom-(co)modules, Hom-Hopf modules and the duality of Hom-algebras, thus enriching the field of Hom-algebra theory.
        As an emerging field within algebra, (Bi)Hom-algebra is closely related to category theory (see, e.g.,  \cite{YW}), combinatorics~\cite{LZZ}, braided group representations (see, e.g., \cite{LWWZ}), and quantum algebras~\cite{CS, LZZ2, SWZ, Sheng, SB, SD}, producing a wealth of research results in recent years.

        The exponentials of infinitesimal characters and their convolutions are known to produce characters with elegant combinatorial counterparts \cite{DBHPS}. In a Hopf algebra, the set of characters forms a group that acquires a natural structure of an (infinite-dimensional) Lie group under certain conditions, which are often met in practice, such as grading~\cite{Du}. All these elements relate to the largest dual that can be considered for a Hopf algebra, known as the Sweedler duality. This duality is, in a sense, related to finite-state automaton theory \cite{DFLL} and algebraic combinatorics \cite{V}.

        In the finite-dimensional case, the structure of the Hom-coalgebra follows directly from the classical duality of the Hom-algebra, and such a conclusion is known. However, this conclusion may not hold in the infinite-dimensional case. This naturally leads to the question of how to get the dual version of such a structure in the infinite-dimensional case. In other words, is it possible to induce the structure of Hom-coalgebras from the general dual of Hom-algebras (resp. from Hom-modules to Hom-comodules) by constructing a correspondence in the infinite-dimensional case? This question was asked by Wang~\cite{SHW} in his conference talk.

        In this paper, we will answer his question by investigating the Sweedler duality for Hom-algebras and their induction into Hom-coalgebras in general. We present the new structure of the Hom-coalgebra induced by the Hom-algebra in the context of the Sweedler duality. Furthermore, we extend our results to Hom-modules and establish the connection between Hom-modules and Hom-comodules.

        The structure of this paper is as follows. In Section~\ref{sect2}, we review some basic concepts and lemmas that discuss the relationship between Hom-algebras and Hom-coalgebras. The structure of Hom-coalgebra under Sweedler duality is given in Theorem~\ref{f}. A morphism of the Hom-coalgebra is derived using the above condition in Corollary~\ref{f1}. In Section~\ref{sect3}, we present a right Hom-comodule of Hom-coalgebra, which is induced by the right module of Hom-algebra under Sweedler duality, as shown in Theorem~\ref{f2}. The (right) Hom-comodule morphism is also given in Theorem~\ref{f3}. In the final section~\ref{sect4}, we develop a Hom-deformed version of binary linearly recursive sequences and use the Sweedler duality to find the recursive sequences of the minimal polynomial $h(x,y)$.

        \paragraph{Notation} Throughout this paper, we assume that the characteristic of the field $\mathbb{K}$ is zero. And we will use the following notation for the mapping role of elements: $f(v):=\langle f, v\rangle$ in which $v$ is any element of some Hom-algebra $G$ and $f$ is an arbitrary element in duality $G^*$. Dual maps of Hom-algebra $(G,\mu,\alpha)$ are \(\mu^{*}\!:G^{*}\!\to G^{*}\!\otimes\!G^{*}\), \(\langle\mu^{*}(f),x\!\otimes\!y\rangle=\langle f,\mu(x\!\otimes\!y)\rangle\), \(\alpha^{*}(f)=f\!\circ\!\alpha\). And for a right Hom-module $(M,\psi,\gamma)$ over $G$ we set \(\psi^{*}\!:M^{*}\!\to M^{*}\!\otimes\!G^{*}\), \(\langle\psi^{*}(\xi),m\!\otimes\!g\rangle=\langle\xi,\psi(m\!\otimes\!g)\rangle\), \(\gamma^{*}(\xi)=\xi\!\circ\!\gamma\). For more details on notation, this paper follows the setting of \cite{G}.
	
\section{Sweedler duality for Hom-algebras}\label{sect2}
        In this section, we will first give an overview of the basic concepts related to the structures of Hom-(co)algebras and the interrelationship between different Hom-(co)algebras (see Definition 2.3 in \cite{DY}). We then move on to introduce the definition of the {\em Sweedler duality}, which differs from the conventional duality (see Chapter 6 in \cite{G}).

        \begin{definition}
        A {\em Hom-associative algebra} is a triple $(G, \mu, \alpha)$, where $G$ is a linear space, $\mu: G \otimes G \rightarrow G$ is a bilinear multiplication with $\mu(g\otimes g')=gg'$. For simplicity, when there is no confusion, we will omit the notation $\mu$. $\alpha: G \rightarrow G$ is a homomorphism. Otherwise, it satisfies the following conditions
       \begin{align}
       \alpha(g)(hk) &= (gh)\alpha(k), \quad \text{ (Hom-associativity)} \label{homAssoc}\\
       \alpha(hk)&=\alpha(h)\alpha(k). \quad \text{ (multiplicativity)} \label{multiplicativity_1}
       \end{align}
       \end{definition}
       
        They can be represented by the following commutative diagrams:
        \begin{equation*}
        \xymatrix{G\otimes G\otimes G\ar@{->}[d]_{\alpha\otimes \mu} & &G\otimes G \ar@{<-}_{\mu\otimes\alpha}[ll]\ar@{->}_{\mu}[d] \\
        G\otimes G & & G, \ar@{<-}_{\mu }[ll] }
        \quad \quad \quad \quad\quad
        \xymatrix{G\otimes G\ar@{->}[d]_{\mu} & &G\otimes G \ar@{<-}_{\alpha\otimes\alpha}[ll]\ar@{->}_{\mu}[d] \\
        G & & G. \ar@{<-}_{\alpha }[ll] }
        \end{equation*}
        
        In the following of this paper, we will abbreviate {\em Hom-associative algebra} to {\em Hom-algebra}, and sometimes we refer to a Hom-algebra $(G, \mu, \alpha)$ as $G$. Now, for two Hom-algebras $(G, \mu, \alpha)$ and $(G', \mu', \alpha')$, let $f: G \to G'$ be a linear map, then $f$ is said to be a morphism of Hom-algebras if it satisfies the following conditions
        \begin{align}\label{s1}
        \mu' \circ (f \otimes f) = f \circ \mu, \quad\quad f \circ \alpha = \alpha' \circ f.
        \end{align}

       \begin{definition}
        Let $(G,\mu,\alpha)$ be a Hom-algebra, then a subspace $H$ of $G$ is a {\em Hom-associative subalgebra} of $G$ if 
        $bc\in H$ and $\alpha(b)\in H$ for any $b,c\in H$.
        In particular, if both $ab\in H$ and $ba\in H$ hold for any $a\in G$ and $b\in H$, we say that $H$ is an {\em ideal} of $G$.
       \end{definition}

       \begin{definition}
       A Hom-coalgebra is a triple $(C, \Delta, \beta)$, where $C$ is a linear space, $\Delta: C \rightarrow C \otimes C$ is a linear map with notation $\Delta(c):=\sum_{(c)}c_{(1)}\otimes c_{(2)}$ for all $c\in C$. And $\beta: C \rightarrow C$ is a homomorphism which satisfies
      \begin{align}
      (\beta \otimes \Delta) \circ \Delta &= (\Delta \otimes \beta) \circ \Delta, \quad \text{ (Hom-coassociativity)} \label{hom-coassoc} \\
      \Delta\circ\beta&=(\beta\otimes\beta)\circ\Delta.  \quad \text{ (comultiplicativity)}\label{multiplicativity_2}
       \end{align}
        \end{definition}

        They can be expressed by the following commutative diagrams:
\begin{equation*}
\xymatrix{C\otimes C\otimes C\ar@{<-}[d]_{\Delta\otimes \beta} & &C\otimes C \ar@{->}_{\beta\otimes \Delta}[ll]\ar@{<-}_{\Delta}[d] \\
C\otimes C & & C, \ar@{->}_{\Delta }[ll] }
\quad\quad\quad\quad\quad
\xymatrix{C\otimes C\ar@{<-}[d]_{\Delta} & &C\otimes C \ar@{->}_{\beta\otimes \beta}[ll]\ar@{<-}_{\Delta}[d] \\
C & & C. \ar@{->}_{\beta }[ll] }
\end{equation*}\par

        Similarly, we will denote a Hom-coalgebra $(C, \Delta, \beta)$ by $C$. A morphism of Hom-coalgebras is a linear map of the underlying $\mathbb K$-linear spaces that commutes with the twisting maps and the comultiplications.

        Now we give some lemmas that describe some fundamental properties of Hom-(co)algebras.

        \begin{lemma}[see \cite{MS2}]\label{s}
        Let $(G,\mu,\alpha)$ be a finite-dimensional Hom-algebra, then $(G^{*},\mu^{*},\alpha^{*})$ is a Hom-coalgebra where $\mu^{*}: G^{*}\rightarrow G^{*}\otimes G^{*}$, $\alpha^{*}(a^{*})=a^{*}\alpha$.
        \end{lemma}

        \begin{lemma}[see \cite{MS2}]\label{qq}
        Let $(G,\mu,\alpha)$ and $(G',\mu',\alpha')$ be two finite-dimensional Hom-algebras, then $f:G\rightarrow G'$ is a morphism of Hom-algebras if and only if $f^{*}:G'^{*}\rightarrow G^{*}$ is a morphism of Hom-coalgebras.
        \end{lemma}

        \begin{lemma}[see \cite{EABC}] \label{zz}
        Let $A$ and $B$ be two linear spaces, $I$ and $J$ be two subspaces of $A$ and $B$ respectively. If there exists a map $\pi:A\otimes B\rightarrow A/I\otimes B/J$, we have $\operatorname{Ker}(\pi)=A\otimes J+B\otimes I$ and $(A\otimes B)/(A\otimes J+B\otimes I)\cong A/I\otimes B/J$.
        \end{lemma}

        \begin{lemma}[see \cite{MS2}]\label{M}
        Let $(G,\mu,\alpha)$ and $(G',\mu',\alpha')$ be two Hom-algebras, the map $f: G\rightarrow G'$ be a morphism of Hom-algebras, $H$ be the ideal of Hom-algebra of $G$, the map $\pi : G\rightarrow G/H$ be a regular full homomorphism. If $\operatorname{Ker}\pi \subseteq \operatorname{Ker} f$, there exists a unique Hom-algebra morphism $g: G/H \rightarrow G^{'}$, which satisfies $g\circ\pi = f$.
        \end{lemma}

        \begin{remark}
        A \emph{regular full homomorphism} means the canonical quotient map  \(  \pi : G \rightarrow G/I\)  associated with an ideal \(I\subseteq G\); that is, a homomorphism of Hom-associative algebras which is surjective (hence a regular epimorphism) and whose kernel equals \(I\).
        \end{remark}

        In order to study the structure of the Hom-coalgebra induced by the infinite-dimensional Hom-algebra and the correspondence of the finite-codimensional ideals in the morphism of Hom-algebras, we will give the definition of Sweedler duality in Definition~\ref{Sweedler_duality_Hom_alg} and the lemma~\ref{b} below.

\begin{definition}\label{Sweedler_duality_Hom_alg}
Let $G$ be a Hom-algebra over $\mathbb K$. The {\em Sweedler duality} of $G$ is based on the classical algebraic duality $G^*$ as follows:
$$
G^{\circ} := \{ f \in G^* \mid \operatorname{Ker} f \supseteq J, \text{ $J$ is a finite-codimensional ideal of } G \}.
$$
\end{definition}

\begin{lemma}\label{b}
    Let $(G,\mu,\alpha)$ and $(G',\mu',\alpha')$ be two Hom-algebras, $f: G\rightarrow G'$ be a morphism of Hom-algebras. If $J$ is a finite-codimensional ideal of $G'$, then its complete inversion $f^{-1}(J)$ is also a finite-codimensional ideal of $G$.
\end{lemma}

\begin{proof}
    According to equation~\eqref{s1}, we obtain the following, which holds for any \( x \in f^{-1}(J) \) and \( a \in G \):
    \begin{align*}
        f(x)\in J \quad \text{and} \quad f(ax)=f(a)f(x)\in J.
    \end{align*}
    Then \( ax \in f^{-1}(J) \). Similarly, we can show that \( xa \in f^{-1}(J) \). For any $x \in f^{-1}(J)$, it can be shown that  $f(x) \in J$ and $\alpha^{\prime}(f(x)) \in J$.

    Observe that $f$ is a morphism of Hom-algebras. Consequently, we have
     $$
     f\circ\alpha(x) = \alpha^{\prime}\circ f(x) \in J,
     $$
     there by establishing $\alpha(x) \in f^{-1}(J)$. We then proceed with $\alpha(f^{-1}(J)) \subseteq f^{-1}(J)$. So $f^{-1}(J)$ is the ideal of Hom-algebra $G$.

    We now demonstrate that $f^{-1}(J)$ is finite-codimensional. Let $g:G\to G'/J$ be a map defined as follows: for any $a\in G$, $a=a_1+a_2$ where $a_1\in f^{-1}(J)$, we have $g(a)=f(a_2)+J$. It is evident that $\operatorname{Ker}(g) = f^{-1}(J)$. Then we can induce the injection $\widetilde{g}:G/f^{-1}(J)\rightarrow G^{'}/J$ ($b+f^{-1}(J)\mapsto f(b)+J$), so we have
    \begin{align*}
       \dim(G/f^{-1}(J))\leq  \dim(G^{'}/J)<+\infty
    \end{align*}
    which implies that the dimension of $G/f^{-1}(J)$ is finite. Consequently, $f^{-1}(J)$ is a finite-codimensional ideal of $G$.
\end{proof}

Inspired by Liu's construction of classical algebras~\cite{LGL}, it is sufficient for us to construct the structure of Hom-coalgebras via the Sweedler duality now.

\begin{theorem}\label{c}
    Let $G$ be a Hom-algebra. Then $G^{\circ}$ is the subspace of $G^{*}$.
\end{theorem}

   \begin{proof}
   It is straightforward to verify that $G^{\circ}$ is closed under $\mathbb K$-multiplication. We proceed to show that $G^{\circ}$ is also closed under addition.

   For any $g,h\in G^{\circ}$, let $J$ and $H$ be two finite-codimensional ideals of $G$ such that $\operatorname{Ker} g\supseteq J$ and $\operatorname{Ker} h\supseteq H$. It is clear that $J\cap H$ is an ideal of the Hom-algebra and $\operatorname{Ker}
    (g+h)\supseteq J\cap H$.

   Our next step is to explain $J\cap H$ is finite-codimensional. Since $H/(J\cap H) \cong (J+H)/J \subseteq G/J$, the quotient $H/(J\cap H)$ is a finite-dimensional linear space.
   Furthermore, since $G/H \cong (G/(J \cap H))/(H/(J \cap H))$, the finite dimension of $G/H$ implies that $dim(G/(J \cap H)) = dim(G/H) + dim(H/(J \cap H))$ is also finite. Consequently, $J \cap H$ is a finite-codimensional ideal of $G$.
   \end{proof}

   \begin{theorem}\label{f}
   Let $(G,\mu,\alpha)$ be a Hom-algebra. Then the Sweedler duality of $G$ is a Hom-coalgebra $(G^{\circ},\Delta,\alpha^{\circ})$, where $\Delta=\mu^{*}|_{G^{\circ}}$, $\alpha^{\circ}=\alpha^{*}|_{G^{\circ}}$ and $\mu^{*}: G^{*}\rightarrow (G\otimes G)^*$.
   \end{theorem}

    \begin{proof}
    Let $f \in G^{\circ}$, we want to show that $\Delta(f) \in G^{\circ} \otimes G^{\circ}$. There exists a finite-codimensional ideal $J$ of $G$ such that $f(J) = 0$. Then we do the formal arithmetic
    \begin{align*}
        \Delta(f)(J \otimes G + G \otimes J) = f\circ\mu(J \otimes G + G \otimes J) = 0,
    \end{align*}
    since $\mu(J \otimes G)\subseteq J$ and $\mu(G \otimes J)\subseteq J$. Let a map $\pi\colon G \otimes G \to G \otimes G / (J \otimes G + G \otimes J)$, then we have $\operatorname{Ker}(\pi) = J \otimes G + G \otimes J \subseteq \operatorname{Ker}(f \mu)$. By Lemma~\ref{M}, this induces a unique morphism $\overline{f \mu}\colon G \otimes G / (J \otimes G + G \otimes J) \to \mathbb K$ satisfying $\overline{f \mu} \circ \pi = f\circ \mu$. That is to say $\overline{f \mu} \in (G \otimes G / (J \otimes G + G \otimes J))^*$.

    According to Lemma~\ref{zz}, we have $G \otimes G / (J \otimes G + G \otimes J) \cong (G/J) \otimes (G/J)$ and $G/J$ is finite, it is easy to see that
    \begin{align*}
        \overline{f\mu} \in(G/J)^* \otimes (G/J)^*.
    \end{align*}

    Let $\overline{f\mu} = \sum_{i,j=1}^n k_{ij} \overline{e_i}^* \otimes \overline{e_j}^*$, where $\overline{e_i}^* \in (G/J)^*$ is the duality base elements corresponding to a canonical map $\Tilde{\pi}\colon G \to G/J$. So for any $a,b \in G$, we have
    \begin{align*}
    \langle f\circ\mu, a \otimes b \rangle &= \langle \overline{f\mu} \circ \pi, a \otimes b \rangle \\
    &= \langle \overline{f\mu} , \pi(a \otimes b) \rangle \\
    &= \langle \overline{f\mu}, \Tilde{\pi}(a) \otimes \Tilde{\pi}(b) \rangle \\
    &= \sum_{i,j=1}^n k_{ij}\langle \overline{e}_i^*, \Tilde{\pi}(a) \rangle \langle \overline{e}_j^*, \Tilde{\pi}(b) \rangle.
    \end{align*}

    Denote $e_i^* = \overline{e}_i^* \circ \Tilde{\pi}$. Because of $\operatorname{Ker} (e_i^*) \supseteq J$, we have $e_i^* \in G^{\circ}$. It is straightforward that
    \begin{align*}
        \Delta(f) = f\circ \mu = \sum\limits_{i,j=1}^n k_{ij} e_i^* \otimes e_j^* .
    \end{align*}
    Therefore, we conclude that $\Delta(f) \in G^{\circ} \otimes G^{\circ}$.

    For any $g \in G^{\circ}$, there exists a finite-codimensional ideal $J$ of $G$ such that $\operatorname{Ker}(g) \supseteq J$. Then, since $\alpha(J)\subseteq J$ and $g(J)=0$, we have
    \begin{align*}
        \langle \alpha^{\circ}(g), J \rangle =\langle \alpha^{*}(g), J \rangle= \langle g, \alpha(J) \rangle= 0.
    \end{align*}
    Hence $\operatorname{Ker}(\alpha^{\circ}(g)) \supseteq J$ follows, which implies that $\alpha^{\circ}(g)\in G^\circ$. 
    
    Using the equation \eqref{multiplicativity_1}, for any $c^*\in G^{\circ}$ and $x,y\in G$, we have
    \begin{align*}
        \langle \Delta\circ\alpha^{\circ}(c^*), x\otimes y\rangle&=\langle c^*, \alpha\circ\mu(x\otimes y)\rangle\\
        &=\langle c^*, \mu\circ(\alpha\otimes\alpha)(x\otimes y)\rangle\\
        &=\langle\Delta(c^*), \alpha(x)\otimes \alpha(y)\rangle\\
        &=\sum_{(c^*)}\langle c^*_{(1)}\otimes c^*_{(2)}, \alpha(x)\otimes \alpha(y)\rangle\\
        &=\sum_{(c^*)}\langle \alpha^{\circ}(c^*_{(1)}), x\rangle\langle \alpha^{\circ}(c^*_{(2)}), y\rangle\\
        &=\langle (\alpha^{\circ}\otimes \alpha^{\circ})(\sum_{(c^*)} c^*_{(1)}\otimes c^*_{(2)}), x\otimes y\rangle\\
        &=\langle (\alpha^{\circ}\otimes \alpha^{\circ})\circ\Delta(c^*), x\otimes y\rangle.
    \end{align*}
    Equivalently speaking, we have $\Delta\circ\alpha^{\circ}=(\alpha^{\circ}\otimes \alpha^{\circ})\circ\Delta$.\par
    
    It is sufficient to verify that
    $(\alpha^{\circ}\otimes \Delta)\circ\Delta=(\Delta \otimes \alpha^{\circ})\circ\Delta.$
     For any $c^{*}\in G^{\circ}$ and $x,y,z\in G$, we have
    \begin{align*}
  \langle (\alpha^{\circ} \otimes \Delta) \circ \Delta(c^*), x \otimes y \otimes z \rangle
  &= \langle (\alpha^{\circ} \otimes \Delta) \circ (\sum_{(c^*)} c^*_{(1)} \otimes c^*_{(2)}) , x \otimes y \otimes z \rangle \\
  &= \sum_{(c^*)} \langle \alpha^{\circ}(c^*_{(1)}) \otimes \Delta(c^*_{(2)}), x \otimes y \otimes z \rangle \\
  &=\sum_{(c^*)} \langle \alpha^{\circ}(c^*_{(1)}), x\rangle\langle\Delta(c^*_{(2)}),  y \otimes z \rangle \\
  &= \sum_{(c^*)} \langle c^*_{(1)}, \alpha(x) \rangle \langle c^*_{(2)}, yz \rangle \\
  &=\sum_{(c^*)} \langle c^*_{(1)}\otimes c^*_{(2)}, \alpha(x)\otimes yz\rangle\\
  &=\langle \Delta(c^*), \alpha(x)\otimes yz\rangle\\
  &= \langle c^*, \alpha(x) yz \rangle.
\end{align*}
    Similarly,
    \begin{align*}
  \langle (\Delta \otimes \alpha^{\circ}) \circ \Delta(c^*), x \otimes y \otimes z \rangle &= \langle c^*, x y \alpha(z) \rangle.
\end{align*}

    Due to the Hom-associativity (\ref{homAssoc}) of $G$, we have $\alpha(x)yz=xy\alpha(z)$. Thus $(\alpha^{\circ}\otimes \Delta)\circ\Delta=(\Delta \otimes \alpha^{\circ})\circ\Delta$. 
\end{proof}

Without specific instructions, we will examine the properties of the Hom-coalgebra $(G^{\circ},\Delta,\alpha^{\circ})$, Sweedler duality of $G$ as described in Theorem~\ref{f}. Based on these considerations, we then derive the morphism of Hom-coalgebras accordingly.

\begin{corollary}\label{f1}
    Let $(G,\mu,\alpha)$ and $(G',\mu',\alpha')$ be two Hom-algebras, $(G^{\circ},\Delta,\alpha^{\circ})$ and $(G'^{\circ},\Delta',\alpha'^{\circ})$ two corresponding objects under the Sweedler duality. The map $f: G\rightarrow G'$ is a morphism of Hom-algebras. Note that the map $f^{\circ} = f^{*}|_{G'^{\circ}}:G'^{\circ}\rightarrow G^{*}$. Thus, $f^{\circ}$ is a morphism of Hom-coalgebras.
\end{corollary}
\begin{proof}
    Our primary goal is to show that \(f^{\circ}(G'^{\circ}) \subseteq G^{\circ}\). For any \(b^{*} \in G'^{\circ}\), since \(J\) is a finite-codimensional ideal of \(G'\) and lies in \(\operatorname{Ker}(b^{*})\), by using Lemma~\ref{b}, we can deduce that \(f^{-1}(J)\) constitutes a finite-codimensional ideal of the Hom-algebra \(G\).

    Notice that
    \begin{align*}
        \langle f^{\circ}(b^{*}), f^{-1}(J) \rangle  = \langle b^{*}, f\circ f^{-1}(J) \rangle 
                          = \langle b^{*}, J \rangle 
                          = 0,
    \end{align*}
    which shows that $\operatorname{Ker}(f^{\circ}(b^{*})) \supseteq f^{-1}(J)$. In other words, $f^{\circ}(b^{*}) \in G^{\circ}$. Consequently, $f^{\circ}(G'^{\circ}) \subseteq G^{\circ}$.

    Using the equation (\ref{s1}), for any $b^{*} \in G^{\prime \circ}$ and $a, c \in G^{\circ}$, we have
    \begin{align*}
        \langle \Delta\circ f^{\circ}(b^{*}), a \otimes c \rangle &= \langle b^{*}, f\circ \mu(a \otimes c) \rangle \\
        &= \langle b^{*}, \mu^{\prime}\circ(f \otimes f)(a \otimes c) \rangle \\
        &= \langle  \Delta'(b^{*}), f(a) \otimes f(c) \rangle \\
        &=\sum_{(b^*)}\langle  b^*_{(1)}\otimes  b^*_{(2)}, f(a) \otimes f(c) \rangle \\
        &=\sum_{(b^*)}\langle  b^*_{(1)}, f(a)\rangle\langle  b^*_{(2)}, f(c) \rangle \\
        &=\sum_{(b^*)}\langle  f^{\circ} (b^*_{(1)}), a\rangle\langle  f^{\circ}(b^*_{(2)}), c \rangle \\
        &=\langle (f^{\circ} \otimes f^{\circ})(\sum_{(b^*)} b^*_{(1)}\otimes  b^*_{(2)}), a \otimes c \rangle\\
        &= \langle (f^{\circ} \otimes f^{\circ})\circ \Delta'(b^{*}), a \otimes c \rangle.
    \end{align*}
It means we obtain that $\Delta\circ f^{\circ} = (f^{\circ} \otimes f^{\circ})\circ \Delta'$.

    For any $b^* \in G'^{\circ}$ and $a \in G$, we have
\begin{align*}
    \langle \alpha^{\circ}\circ f^{\circ}(b^*), a \rangle = \langle b^*, f\circ \alpha(a) \rangle = \langle b^*, \alpha'\circ f(a) \rangle
    = \langle f^{\circ}\circ \alpha'^{\circ}(b^*), a \rangle.
\end{align*}
Therefore, it follows that $\alpha^{\circ}\circ f^{\circ} = f^{\circ}\circ \alpha'^{\circ}$, which is the desired result.
\end{proof}

Using the Sweedler duality that we constructed, ordinary Hom-associative algebras naturally get an infinite-dimensional correspondence. We present two as examples in the following.

\begin{example}[Polynomial Hom-algebra]
Let $\mathbb K$ be a field and we consider the Hom-algebra $(\mathbb K[x],\cdot,\alpha)$ where $\mathbb K[x]$ is the set of the polynomials, $\alpha$ is a linear map acting on the indeterminate $x$ with the property $\alpha(x^m)=\alpha(x)^m$ for any $m\in N$ and the bilinear operator $\cdot:\mathbb K[x]\rightarrow \mathbb K[x]$ with the notation 
\[
  x^{m}\cdot x^{n}=\alpha(x^m)\alpha(x^n)=\alpha(x)^{\,m+n}\quad (m,n\ge 0).
\]
A linear functional $\varphi \in A^{*}=\operatorname{Hom}_{\mathbb K}(\mathbb K[x],\mathbb K)$ lies in the \emph{Sweedler duality}
\[
  A^{\circ}
  \;=\;
  \bigl\{\varphi\in A^{*}\mid
         \exists\, N\ge 0 ~|~
         \varphi\!\bigl((x^{N+1})\bigr)=0
     \bigr\}
\]
iff it annihilates some finite-codimensional ideal $(x^{N+1})$.  
Hence $(A^{\circ},\Delta,\mu)$ is the Hom-coalgebra induced by the Sweedler duality where the coordinate functionals span the map $\mu$
\[
  \mu(d_n)(x^m)=d_{n}(\alpha(x^{m})) =d_n(\alpha(x)^m)= \delta_{n,m},
  \qquad n,m\ge 0,
\]
dualising the multiplication $\cdot\colon \mathbb K[x]\!\otimes\!\mathbb K[x]\to\mathbb K[x]$ yields a comultiplication
\[
  \Delta(d_{n})
  \;=\;
  \sum_{i=0}^{n} \mu(d_{i})\otimes \mu(d_{\,n-i}),
  \qquad n\ge 0.
\]
\end{example}

\begin{example}[Tensor Hom-algebra $T(V)$ on a finite-dimensional vector space]
Let $V$ be a finite-dimensional vector space over a field $\mathbb K$. Its tensor algebra
\[
T(V)=\mathbb K\oplus\bigoplus_{i=1}^{\infty}V^{\otimes i}=\mathbb K \;\oplus\; V \;\oplus\; (V^{\otimes 2}) \;\oplus\;\cdots
\]
is a Hom-algebra $(T(V),\cdot,\alpha)$ where the linear map $\alpha:T(V)\rightarrow T(V)$ with the notation $\alpha(V^{\otimes m})=\alpha(V)^{\otimes m}$and the multiplication $\cdot:T(V)\otimes T(V)\rightarrow T(V)$ are given by \emph{tensor Hom-concatenation} such that for any $a,b\in \mathbb K$ and $A\in V^{\otimes p},B\in V^{\otimes q},C\in V^{\otimes r}$,
\begin{align*}
    &A\cdot B=\alpha(A)\otimes \alpha(B),\\
    &(aA+bB)\cdot C=a(A\cdot C)+b(B\cdot C).
\end{align*}

For any word $w = k\otimes v_{1}\otimes\cdots\otimes v_{n}\,(v_{i}\in V)$
write $|w|=n+1$.
Define the evaluation functionals
\[
  d_{w}\colon T(V) \longrightarrow \mathbb K,
  \qquad
  d_{w}(w') = \delta_{w,w'},
\]
which clearly annihilate the finite-codimensional ideal
\(
  I_{n} = \bigoplus_{m>n} V^{\otimes m}.
\)
These $d_{w}$ form a basis of the Sweedler duality $(A^{\circ},\Delta,\gamma)$ where the map $\gamma$ with the calculation 
\[
\gamma(d_w)(aA+bB)=d_w(a\alpha(A)+b\alpha(B))=a\delta_{w,p}+b\delta_{w,q},
\]
dualising the multiplication $\cdot:T(V)\otimes T(V)\rightarrow T(V)$ again,
\[
  \Delta(d_{w})
  \;=\;
  \sum_{uv = w} \mu(d_{u})\otimes \mu(d_{v}),
\]
where the sum runs over every decomposition of $w$ into a prefix $u$ and a suffix $v$. 
\end{example}

 \section{Sweedler duality for Hom-modules}\label{sect3}
        In this section, we will introduce the structure of Hom-(co)modules and the morphisms of Hom-(co)modules (see definition 4.1 in~\cite{DY}). Based on these structures, the Sweedler duality of Hom-(co)modules will also be treated. In subsequent discourse, we utilize $\cdot$ to represent the algebraic subset Cartesian product, maintaining equivalence to unsigned operations, thereby clarifying operational sequencing in intricate computations, such as $M\cdot J=MJ$.

    \begin{definition}
        Let $(G, \mu, \alpha)$ be a Hom-algebra. A right Hom-module of $G$ is a triple $(M, \psi, \gamma)$ where $M$ is a linear space, a map $\psi: M \otimes G \rightarrow M$ with $\psi(m\otimes g)=m\cdot g$ and a map $\gamma: M \rightarrow M$ satisfying:
        \begin{align*}
            \psi\circ(\psi\otimes\alpha)&=\psi\circ(\gamma\otimes\mu),\quad \text{(Hom-associativity)}\\
        \psi\circ(\gamma\otimes\alpha)&=\gamma\circ\psi.\quad \quad \quad \quad
        \text{(multiplicativity)}
        \end{align*}
        
    \end{definition}

    We will refer to a right Hom-module $(M,\psi,\gamma)$ as $M$.
     Let $(M,\psi,\gamma)$ and $(N,\Tilde{\psi},\Tilde{\gamma})$ be two right Hom-modules of $G$. A map $\sigma: M \rightarrow N$ is said to be a morphism of right Hom-modules when the following conditions hold:
     \begin{align}\label{s7}
    \sigma\circ \psi = \Tilde{\psi} \circ (\sigma\otimes\alpha), \quad \quad
    \Tilde{\gamma} \circ \sigma = \sigma \circ \gamma.
       \end{align}

\begin{definition}
    Let $(C, \Delta, \beta)$ be a Hom-coalgebra. A right Hom-comodule of $C$ is a triple $(A, \phi, \epsilon)$ where $A$ is a linear space, a map $\phi: A \rightarrow A\otimes C$ with $\psi(a)=\sum_{(a)}a_{(0)}\otimes a_{(1)}$ and a map $\epsilon: C \rightarrow C$ satisfying:
    \begin{align*}
     (\phi\otimes\beta)\circ\phi&=(\epsilon\otimes\Delta)\circ\phi,\quad \text{(Hom-coassociativity)}\\
    (\epsilon\otimes\beta)\circ\phi&=\phi\circ\epsilon.\quad \quad \quad \quad \quad \text{(comultiplicativity)}    
    \end{align*}
    
\end{definition}
    We will refer to a right Hom-comodule $(A, \phi, \epsilon)$ as $A$. Let $(A, \phi, \epsilon)$ and $(A', \phi', \epsilon')$ be two right Hom-comodule of $C$. A map $f:A\rightarrow A'$ is called a morphism of right Hom-comodules if it satisfies:
    $$
    \phi'\circ f= (f\otimes \beta)\circ\phi,
    \quad \quad
    \epsilon'\circ f=f\circ\epsilon.
    $$

\begin{definition}
    The \textit{Sweedler duality} of the right Hom-module $M$ based on the ordinary duality $M^*$ is defined by
    $$M^{\circ} := \{ f \in M^{*} \mid f(M\cdot I) = 0, I \text{ is a  finite-codimensional ideal of  G  } \}.$$
\end{definition}

Then we will give the isomorphism in dual spaces (see~\cite{LGL})  and
the structure of Hom-comodule induced by the universal duality as follows.
\begin{lemma}\label{iso}
    Let $U$ and $V$ be two linear spaces. Denote $\rho:U^*\otimes V^*\rightarrow (U\otimes V)^*$ with corresponding rule $\rho(f\otimes g)(x\otimes y)=f(x)g(y)$, for any $f\in U^*$, $g\in V^*$, $x\in U$ and $y\in V$. If $\operatorname{dim} U<+\infty$ or $\operatorname{dim} V<+\infty$, $\rho$ is an isomorphism of linear spaces.
\end{lemma}

\begin{lemma} \label{q}
    Let $(G, \mu, \alpha)$ be a finite-dimensional Hom-algebra and $(M, \psi, \gamma)$ be a right Hom-module of $G$. Then, $(M^{*}, \psi^{*}, \gamma^{*})$ is a right Hom-comodule of a Hom-coalgebra $(G^{*}, \mu^{*}, \alpha^{*})$, where $\psi^{*}: M^{*} \rightarrow M^{*} \otimes G^{*}$ and $\gamma^{*}:M^*\rightarrow M^*$.
\end{lemma}

\begin{proof}
    Applying Lemma~\ref{s}, we can find $(G^{*}, \mu^{*}, \alpha^{*})$ is a Hom-coalgebra. 
    For any $m^*\in M^*$, $m\in M$ and $g\in G$, we have
    \begin{align*}
        \langle \psi^*\circ\gamma^*(m^*), m\otimes g\rangle&=\langle m^*, \gamma\circ\psi(m\otimes g)\rangle\\
        &=\langle m^*, \psi\circ(\gamma\otimes \alpha)(m\otimes g)\rangle\\
        &=\langle \psi^*(m^*), \gamma(m)\otimes \alpha(g)\rangle\\
        &=\sum_{(m^*)}\langle m^*_{(0)}\otimes m^*_{(1)}, \gamma(m)\otimes \alpha(g)\rangle\\
        &=\sum_{(m^*)}\langle \gamma^*(m^*_{(0)}), m\rangle\langle \alpha^*(m^*_{(1)}), g\rangle\\
        &=\sum_{(m^*)}\langle (\gamma^*\otimes \alpha^*)(m^*_{(0)}\otimes m^*_{(1)}), m\otimes g\rangle\\
        &=\langle (\gamma^*\otimes \alpha^*)\circ\psi^*(m^*),  m\otimes g\rangle.
    \end{align*}
    Therefore we have $\psi^*\circ\gamma^*=(\gamma^*\otimes \alpha^*)\circ\psi^*$.

    Then it is sufficient for us to check that $(\psi^*\otimes\alpha^*)\circ\psi^*=(\gamma^*\otimes\mu^*)\circ\psi^*$. For any $a^*\in M^*$ and $x\in M, y,z\in G$,
    \begin{align*}
        \langle(\psi^*\otimes\alpha^*)\circ\psi^*(a^*), x\otimes y\otimes z\rangle&=\langle(\psi^*\otimes\alpha^*)(\sum_{(a^*)}a^*_{(0)}\otimes a^*_{(1)}), x\otimes y\otimes z\rangle\\
        &=\sum_{(a^*)}\langle \psi^*(a^*_{(0)}), x\otimes y \rangle\langle \alpha^*(a^*_{(1)}), z \rangle\\
        &=\sum_{(a^*)}\langle a^*_{(0)}, \psi(x\otimes y) \rangle\langle a^*_{(1)}, \alpha(z)\rangle\\
        &=\langle \psi^*(a^*), (\psi\otimes\alpha)(x\otimes y\otimes z)\rangle\\
        &=\langle a^*, \psi\circ(\psi\otimes\alpha)(x\otimes y\otimes z)\rangle.
    \end{align*}
    Similarly, we have $\langle(\gamma^*\otimes\mu^*)\circ\psi^*(a^*), x\otimes y\otimes z\rangle=\langle a^*, \psi\circ(\gamma\otimes\mu)(x\otimes y\otimes z)\rangle$. According to the equation~\eqref{s7}, It can be shown that $(\psi^*\otimes\alpha^*)\circ\psi^*=(\gamma^*\otimes\mu^*)\circ\psi^*$. This completes the proof of the lemma.
\end{proof}

Now, we will construct the structure of the Hom-comodule via the Sweedler duality procedure.

\begin{theorem}\label{c'}
    Let $G$ be a Hom-algebra and $M$ be a Hom-module. Then $M^{\circ}$ is the subspace of $M^*$.
\end{theorem}

\begin{proof}
    Firstly, it can be observed that $M^{\circ}$ is closed under multiplication. So it is sufficient for us to show that $M^{\circ}$ is closed under addition.

    Let $f, g \in M^{\circ}$, there exist finite-codimensional ideals $I$ and $J$ of $G$ such that $f(M\cdot I) = 0$ and $g(M\cdot J) = 0$. By Theorem \ref{c}, we have $I \cap J$ is an ideal of $G$ and $G/(I \cap J)$ is finite-dimensional. Then, we have
        $(f+g)(M\cdot (I\cap J)) = 0$.
      And we may conclude that $f+g \in M^{\circ}$, thereby completing the proof.
\end{proof}

\begin{theorem} \label{f2}
    Let $(M,\psi,\gamma)$ be a right Hom-module of a Hom-algebra $(G,\mu,\alpha)$ where $\alpha$ is an automorphism map of $G$, then under the Sweedler duality constrcution,  $(M^{\circ}, \psi^{\circ},\gamma^{\circ})$ is a right Hom-comodule of $G^{\circ}$ where $\psi^{\circ}=\psi^*|_{M^{\circ}}$, $\gamma^{\circ}=\gamma^*|_{M^{\circ}}$ and $\psi^{*}:M^{*}\rightarrow (M\otimes G)^*$.
\end{theorem}

\begin{proof}
Noticing that $f\in M^{\circ}$, we aim to verify that $\psi^{\circ}(f)\in M^{\circ}\otimes G^{\circ}$. There exists a finite-codimensional ideal $J$ of $G$ such that $f(M\cdot J) = 0$. Since $\alpha$ is the automorphism map of $G$, $\gamma(M)\subseteq M$ and $JG\subseteq J$, we have
    \begin{align*}
        \psi^{\circ}(f)(M\cdot J \otimes  G + M \otimes J) &=
         \psi^{\circ}(f)(M\cdot J \otimes \alpha(G) + M \otimes J) \\
        &=f\circ\psi(M\cdot J \otimes \alpha(G) + M \otimes J) \\
        &= f(\psi(M\cdot J \otimes \alpha(G)) + \psi(M \otimes J)) \\
        &= f(\psi(\gamma(M) \otimes JG)+\psi(M\otimes J))\\
        &= 0.
    \end{align*}
    Denote $\pi : M \otimes G \rightarrow M \otimes G / (M\cdot J \otimes G + M \otimes J)$ and $\operatorname{Ker}(\pi) = M\cdot J \otimes G + M \otimes J \subseteq \operatorname{Ker}(f\circ \psi).$
    Using Lemma~\ref{zz}, we can induce a unique morphism $\overline{f\psi} : M \otimes G / (M\cdot J \otimes G + M \otimes J) \rightarrow \mathbb K$ satisfying $\overline{f\psi}\circ \pi=f\circ\psi$, that is to say,
    \begin{align*}
    \overline{f\psi} \in \left( M \otimes G / (M\cdot J \otimes G + M \otimes J) \right)^*.
    \end{align*}

    According to the Lemma~\ref{iso}, we can get $M \otimes G / (M\cdot J \otimes G + M \otimes J)\cong (M / M\cdot J) \otimes (G / J)$ with finite-dimensional linear space $G / J$, so we have
\begin{align*}
    \overline{f \psi} \in (M / M\cdot J)^* \otimes (G / J)^*.
\end{align*}

    Denote two canonical maps $\pi_1 : M \rightarrow M/M\cdot J$ and $\pi_2 : G \rightarrow G/J$. And then we denote $\overline{f \psi} = \sum^n_{i,j=1}k_{ij}\overline{d_i}^* \otimes \overline{e_j}^*$ where $\overline{d_i}^* \in (M/M\cdot J)^*$, $\overline{e_j}^* \in (G/J)^*$ are the dual base elements. After a similar discussion as in Theorem~\ref{f}, for any $x\in M$ and $y\in G$, we can get that
    $$
    \langle f\circ\psi, x\otimes y\rangle=\sum_{i,j=1}^n k_{ij}\langle\overline{d_i}^*,\pi_1(x)\rangle\langle\overline{e_j}^*,\pi_2(y)\rangle.
    $$
    Denote $d_i^*=\overline{d_i}^*\circ\pi_1$ and $e_j^*=\overline{e_j}^*\circ\pi_2$. Since $d_i^*(M\cdot J)=0$ and $\operatorname{Ker}(e_j^*)\supseteq J$, we have $d_i^*\in M^{\circ}$ and $e_j^*\in G^{\circ}$. Thus, we have
    $$
    \psi^{\circ}(f)=f\circ\psi=\sum_{i,j=1}^n k_{ij} d_i^*\otimes e_j^*.
    $$  Consequently, $\psi^{\circ}(f) \in M^{\circ} \otimes G^{\circ}$.\par
    For any $m\in M^{\circ}$, there exists a finite-codimensional ideal $J$ of $G$ such that $m(M\cdot J)=0$. Then we have
    \begin{align*}
        \langle \gamma^{\circ}(m), M\cdot J\rangle=\langle m ,\gamma(M\cdot J)\rangle=0,
    \end{align*}
    since $\gamma(M\cdot J)=\gamma(M)\cdot \alpha(J)\subseteq M\cdot J$. Then we have $\gamma^{\circ}(m)(M\cdot J)=0$, which implies $\gamma^{\circ}(m)\in M^{\circ}$. Thus, $\gamma^{\circ}(M^{\circ})\subseteq M^{\circ}$. 
    Combining Theorem~\ref{c}, Theorem~\ref{c'}, and Lemma~\ref{q} together, we deduce that $\psi^{\circ}\circ\gamma^{\circ}=(\gamma^{\circ}\otimes\alpha^{\circ})\circ\psi^{\circ}$ and $(\psi^{\circ}\otimes\alpha^{\circ})\circ\psi^{\circ}=(\gamma^{\circ}\otimes\Delta)\circ\psi^{\circ}$. Hence, $M^{\circ}$ is a right Hom-comodule of $G^{\circ}$.
\end{proof}

\begin{remark}
In the proof of Theorem~\ref{f2} we must insist that the twisting map
$\alpha:G\to G$ is \emph{bijective}.  
Surjectivity gives $\alpha(G)=G$ and, in particular, $\alpha(J)=J$ for every ideal
$J\subseteq G$.  Hence
\[
  M\!\cdot\!J\otimes G
  \;=\;
  M\!\cdot\!J\otimes\alpha(G),
\]
which implies
\[
  \psi^{\circ}(f)\bigl(M\!\cdot\!J\otimes G + M\otimes J\bigr)=0.
\]
This step is essential for establishing that $\psi^{\circ}(f)$ lies in
$M^{\circ}\otimes G^{\circ}$.
\end{remark}

In the following, we will study the properties of the Hom-comodule $(M^{\circ},\psi^{\circ},\gamma^{\circ})$ proposed in Theorem~\ref{f2}. Based on the above considerations, we continue to derive the morphism of Hom-comodules as follows.

\begin{theorem} \label{f3}
    Let $(M,\psi,\gamma)$ and $(N,\psi',\gamma')$ be two right Hom-modules over Hom-algebra $(G,\mu,\alpha)$ where $\alpha$ is an automorphism map of $G$, $(M^{\circ},\psi^{\circ},\gamma^{\circ})$ and $(N^{\circ},\psi'^{\circ},\gamma'^{\circ})$ be two corresponding objects under the Sweedler duality over $G^{\circ}$. Let $\sigma : M \rightarrow N$ be a morphism of right Hom-modules. Note that $\sigma^{\circ} :=\sigma^*|_{N^{\circ}}:N^{\circ}\rightarrow M^{*}$. Then $\sigma^{\circ}$ is a morphism of Hom-comodules over $G^{\circ}$.
\end{theorem}

\begin{proof}
    For all $f \in N^{\circ}$, let $J$ be a finite-codimensional ideal of $G$ such that $f(N\cdot J) = 0$. By the properties of the morphism of the right Hom-modules, we have
    \begin{align*}
    \sigma^{\circ}(f)(M\cdot J) &= \sigma^{\circ}(f)\circ\psi(M \otimes J) \\
    &= f(\sigma\circ \psi(M \otimes J)) \\
    &= f(\psi' \circ (\sigma \otimes \alpha)(M \otimes J)) \\
    &= f(\sigma(M)\cdot \alpha(J))\\
    &= 0,
    \end{align*}
    since $\sigma(M)\subseteq N$ and $\alpha(J)\subseteq J$. Hence $\sigma^{\circ}(N^{\circ}) \subseteq M^{\circ}$.
    
    For any $b^*\in N^{\circ}$ and $x\in M, y\in G$, we have
    \begin{align*}
        \langle(\sigma^{\circ}\otimes\alpha^{\circ})\circ\psi'^{\circ}(b^*), x\otimes y\rangle&=\sum_{(b^*)}\langle(\sigma^{\circ}\otimes\alpha^{\circ})(b^*_{(0)}\otimes b^*_{(1)}), x\otimes y\rangle\\
        &=\sum_{(b^*)}\langle \sigma^{\circ}( b^*_{(0)}), x\rangle\langle \alpha^{\circ} (b^*_{(1)}), y\rangle\\
        &=\langle \psi'^{\circ}(b^*),  (\sigma\otimes\alpha)(x\otimes y)\rangle\\
        &=\langle b^*, \psi'\circ(\sigma\otimes\alpha)(x\otimes y)\rangle\\
        &=\langle b^*, \sigma\circ\psi(x\otimes y)\rangle\\
        &=\langle\psi^{\circ}\circ\sigma^{\circ}(b^*), x\otimes y\rangle.
    \end{align*}
    Thus we obtain $(\sigma^{\circ}\otimes\alpha^{\circ})\circ\psi'^{\circ}=\psi^{\circ}\circ\sigma^{\circ}$.
    And then for all $c^*\in N^{\circ}$ and $x\in M$,
    \begin{align*}
        \langle \gamma^{\circ}\circ\sigma^{\circ}(c^*), x\rangle
        =\langle c^*, \sigma\circ\gamma(x)\rangle
        =\langle c^*, \gamma'\circ\sigma(x)\rangle
        =\langle \sigma^{\circ}\circ\gamma'^{\circ}(c^*), x\rangle.
    \end{align*}
    Thus it follows that  $\gamma^{\circ}\circ\sigma^{\circ}=\sigma^{\circ}\circ\gamma'^{\circ}$. 
\end{proof}

    In the special case, when $M=G$, we have $M\cdot J=GJ\subseteq J$. Then the Sweedler duality of the Hom-module $M$ reduces to the duality between Hom-algebras and Hom-coalgebras.

\section{Hom-deformed binary linearly recursive sequences}\label{sect4}

The relationship between binary linearly recursive sequences and Sweedler duality, as well as their constructions, is an interesting problem that attracts many researchers. We do not emphasise this here; see ~\cite{A, GW} and their references, for example. While constructing a Hom-deformed version of binary linearly recursive sequences has been challenging due to the absence of a well-defined Sweedler duality for Hom-algebras, we now utilize the previously established definition of Sweedler duality for Hom-algebras to accomplish this task.

Now consider a polynomial algebra \(A=\mathbb{K}[x]\) in one variable \(x\). Denote a linear map $\alpha:A\rightarrow A$ with the notation $\alpha(x^i)=k^ix^i$ where $i\in\mathbb{N}$ and $k$ is a non-zero constant in the field. Then we can construct a Hom-algebra $(A,\mu,\alpha)$ with the product $\mu:A\otimes A\rightarrow A$ such that $\mu(p_1[x]\otimes p_2[x])=\alpha(p_1[x])\alpha(p_2[x])$. On the other hand, we can identify an element \(f\) in conventional duality \(A^{*}\) with the sequences \((f_{n})_{n\geq 0}=(f_{0},f_{1},f_{2},\cdots)\), where \(f_{n}=f(x^{n})\) for \(n\geq 0\). Using Theorem~\ref{f}, under the Sweedler duality of \(A\), we have a Hom-coalgebra \(A^{\circ}=\{f\in A^{*}\mid f(J)=0,\text{ $J$ is a finite-codimensional ideal of $A$ }\}\). 

We now consider the binary polynomial algebra \(B=\mathbb{K}[x,y]\) with a group-like element \(x\) and with a \(y\) \((x,1)\)-primitive element. Let $\alpha:B\rightarrow B$ be a linear map satisfying $\alpha(x^iy^j)=k^{i+j}x^iy^j$, where $i,j\in \mathbb N$ and $k$ is a non-zero constant in the field $\mathbb K$. Then we can see that $(B,\mu,\alpha)$ a Hom-algebra where the product $\cdot:B\otimes B\rightarrow B$ with the notation $p_1[x,y]\cdot p_2[x,y]=\alpha(p_1[x,y])\alpha(p_2[x,y])$.  With the Theorem~\ref{f}, the Sweedler duality \(B^{\circ}\) is a Hom-coalgebra. We identify each \(f\) in the conventional duality \(A^{*}\) as a binary-sequence \((f_{i,j})\) for \(i,j>0\), where \(f_{i,j}=f(x^{i}y^{j})\). A row of such a binary-sequence is a sequence \(\{f_{i,p}\mid p\geq 0\}\) for a fixed \(i\geq 0\), which we say is parallel to the \(y\)-axis, or a sequence \(\{f_{p,j}\mid p\geq 0\}\) for a fixed \(j\geq 0\), which we say is parallel to the \(x\)-axis.

Let \(f\) be in \(B^{\circ}\), \(f(J)=0\) for the finite-codimensional ideal \(J\) of \(B\). For each \(i,j\), the powers of \(x(resp. y)\) span a finite-dimensional space in \(A/J\), so there is a minimal monic \(h_{i}(x)(resp. h_{j}(y))\) in \(\mathbb{K}[x,y]\) such that each row of \(f\) parallel to the \(y(resp. x)\)-axis satisfies \(h_{i}(x)(resp. h_{j}(y))\). Thus, \(J\) contains the finite-codimensional elementary ideal \(\Gamma\) generated by \(h_{i}(x)h_{j}(y)\).

In the subsequent two subsections, we will introduce finite-codimensional ideals with recursive relations and quantum binary sequences in the Hom-deformed version. Since our chosen map $\alpha$ ensures that the Hom-associator satisfies the pentagon axiom in tensor categories, the resulting structures of these two objects in the Hom-version remain proximate to Wang's original results presented in his seminal work \cite{GW}. Let us now proceed to a formal exposition of these two constructs.

\subsection{Quantum binary sequences}
Let \( C = \mathbb{K}_q[x, y] \) with \( yx = qxy \), where \( 0 \neq q \in \mathbb{K} \). Then, \( f \in C^* \) is regarded as the binary sequences \( (f_{m,n})_{m,n \geq 0} = (f_{0,0}, f_{0,1}, \dots, f_{0,n}, f_{1,0}, \dots, f_{1,n}, \dots, f_{m,0}, \dots, f_{m,n}, \dots) \), where
\[
f_{m,n} = f(x^m y^n) = q^{-mn} f(y^n x^m) = q^{-mn} f_{n,m},
\]
for all \( m, n \geq 0 \). We call them the \( q \)-binary sequences. Then we will give some polynomials based on the parameter $q$:
\begin{enumerate}
    \item Given a \(q\neq 0\) in k and an integer \(n>0\), one knows
\[(n)_{q}=\frac{q^{n}-1}{q-1}=1+q+\cdots+q^{n-1}.\]
The \(q\)-factorial of \(n\) is given by \((0)!_{q}=1\), and if \(n>0\),
\[(n)!_{q}=(1)_{q}(2)_{q}\cdots(n)_{q}=\frac{(q-1)(q^{2}-1)\cdots(q^{n}-1)}{(q-1)^{n}}.\] It is a polynomial in \(q\) with coefficients in \(\mathbb{Z}\). Moreover, it has value at \(q=1\) equal to \(n!\).
    \item The Gaussian polynomials is given by for \(0\leq i\leq n\)
\[\left(\begin{array}{c}n\\ i\end{array}\right)_{q}=\frac{(n)_{q}!}{(i)_{q}!(n-i)_{q}!}.\]
    \item Let \(x\) and \(y\) be variables subject to the quantum plane relation \(yx=qxy\). Then, for any \(n>0\), with the product $\cdot$, we have
\[(x+y)^{n}=\sum_{0\leq i\leq n}\binom{n}{i}_{q}k^{\frac{(n-1)(n+2)}{2}}x^{i}y^{n-i}.\]
If \(k=1=q\), the coefficient of the ordinary binomial remains \(\binom{n}{i}\).
\end{enumerate}

Assuming \( C \) is a Hom-algebra $(C,\cdot,\alpha)$, where $\alpha:C\rightarrow C$ is a linear map with $\alpha(x^iy^j)=k^{i+j}x^iy^j$, $i,j\in \mathbb N$ and $k$ is a non-zero constant in the field $\mathbb K$, the product $\cdot:C\otimes C\rightarrow C$, $p_q[x,y]\cdot\overline{p}_q[x,y]=\alpha(p_q[x,y])\alpha(\overline{p}_q[x,y])$. So the quantum convolution product on \( C^* \) is given by \( f_{m,n} \ast_q g_{m,n} = h_{m,n} \), where
\[
h_{m,n} = \sum_{0 \leq k \leq n} \binom{n}{k}_q f_{m+k,n-k} \otimes g_{m,k} \quad \text{for } m, n \geq 0.
\]

As a final remark of this subsection, the interested reader may consider the problem of whether the space of Hom-binary linearly recursive sequences is closed or not under the quantum convolution product, based on the version we have obtained.

\subsection{Finite-codimensional ideals and recursive relations}
Now we know that the finite-codimensional ideal $J$ contains the monic binary polynomial $h(x,y)$. To examine Sweedler duality in more detail, we will study $h(x,y)$ contained in $J$.  By a finite-codimensional ideal \( J \) of \( C = \mathbb{K}_q[x,y] \), we mean that a non-zero ideal generated by a monic binary polynomial:
\begin{align*}
h(x, y) &= x^r y^s - h_{1,0} x^{r-1} y^s - \dots - h_{r,0} y^s\\
&- h_{0,1} x^r y^{s-1} - h_{1,1} x^{r-1} y^{s-1}- \dots - h_{r,1} y^{s-1}\\
&- h_{0,2} x^r y^{s-2} - h_{1,2} x^{r-1} y^{s-2}- \dots - h_{r,2} y^{s-2}\\ 
&\cdots\quad \cdots\quad \cdots\quad \cdots\\
&- h_{0,s} x^r - h_{1,s} x^{r-1} - \dots - h_{r,s}.
\end{align*}

We have the following cases due to the condition \( f(J) = 0 \):
\begin{itemize}
    \item \textbf{Case 1:} If \( f((x^{m-r}\cdot h(x, y))\cdot y^{n-s}) = 0 \), then we have a binary linearly recursive sequence \( f = (f_{m,n})_{m \geq r, n \geq s} \) satisfying the recursive relation \( h(x, y) \), where
\begin{align*}
f_{m,n} &= h_{1,0} f_{m-1,n} + h_{2,0} f_{m-2,n} + \dots + h_{r,0} f_{m-r,n}\\
&+ h_{0,1} f_{m,n-1} + h_{1,1} f_{m-1,n-1} + \dots + h_{r,1} f_{m-r,n-1}\\
&+ h_{0,2} f_{m,n-2} + h_{1,2} f_{m-1,n-2} + \dots + h_{r,2} f_{m-r,n-2}\\
&\cdots\quad \cdots\quad \cdots\quad \cdots\\
&+ h_{0,s} f_{m,n-s} + h_{1,s} f_{m-1,n-s} + \dots + h_{r,s} f_{m-r,n-s}.
\end{align*}
    \item \textbf{Case 2:} If \( f((x^{m-r}\cdot y^{n-s})\cdot h(x, y)) = 0 \), then we have a parameterized binary linearly recursive sequence \( f = (f_{m,n})_{m \geq r, n \geq s} \) satisfying the recursive relation \( h(x, y) \), where
\begin{align*}
f_{m,n} &= q^{-(n-s)} h_{1,0} f_{m-1,n} + q^{-2(n-s)} h_{2,0} f_{m-2,n} + \dots + q^{-r(n-s)} h_{r,0} f_{m-r,n}\\
&+ h_{0,1} f_{m,n-1} + q^{-(n-s)} h_{1,1} f_{m-1,n-1} + q^{-2(n-s)} h_{2,1} f_{m-2,n-1} + \dots + q^{-r(n-s)} h_{r,1} f_{m-r,n-1}\\
&+ h_{0,2} f_{m,n-2} + q^{-(n-s)} h_{1,2} f_{m-1,n-2}+ q^{-2(n-s)} h_{2,2} f_{m-2,n-2} + \dots + q^{-r(n-s)} h_{r,2} f_{m-r,n-2}\\
&\cdots\quad \cdots\quad \cdots\quad \cdots\\
&+ h_{0,s} f_{m,n-s} + q^{-(n-s)} h_{1,s} f_{m-1,n-s}+ q^{-2(n-s)} h_{2,s} f_{m-2,n-s} + \dots + q^{-r(n-s)} h_{r,s} f_{m-r,n-s}.
\end{align*}
    \item \textbf{Case 3:} If \( f((h(x, y)\cdot x^{m-r})\cdot y^{n-s}) = 0 \), then we have a parameterized binary linearly recursive sequence \( f = (f_{m,n})_{m \geq r, n \geq s} \) satisfying the recursive relation \( h(x, y) \), where
\begin{align*}
f_{m,n} &= h_{1,0} f_{m-1,n} + h_{2,0} f_{m-2,n} + \dots + h_{r,0} f_{m-r,n}\\
&+ q^{-(m-r)} h_{0,1} f_{m,n-1} + q^{-(m-r)} h_{1,1} f_{m-1,n-1} + \dots + q^{-(m-r)} h_{r,1} f_{m-r,n-1}\\
&+q^{-2(m-r)} h_{0,1} f_{m,n-2} + q^{-2(m-r)} h_{1,2} f_{m-1,n-2} + \dots + q^{-2(m-r)} h_{r,2} f_{m-r,n-2}\\ 
&\cdots\quad \cdots\quad \cdots\quad \cdots\\
&+ q^{-s(m-r)} h_{0,s} f_{m,n-s}+ q^{-s(m-r)} h_{1,s} f_{m-1,n-s} + \dots + q^{-s(m-r)} h_{r,s} f_{m-r,n-s}.
\end{align*}
\end{itemize}

\begin{remark}
\begin{enumerate}
    \item Let \( h'_{i,j} = q^{-i(n-s)} h_{i,j} \) in Case 1. Consequently, we obtain a new monic binary polynomial \( h'(x, y) \), leading to a new binary linearly recursive sequence \( f' = (f'_{m,n})_{m \geq r, n \geq s} \) that satisfies the recursive relation defined by \( h'(x, y) \).
    \item Similarly, define \( h''_{i,j} = q^{-j(m-r)} h_{i,j} \) for Case 2. This leads us to a new monic binary polynomial \( h''(x, y) \), which in turn generates a new binary linearly recursive sequence \( f'' = (f''_{m,n})_{m \geq r, n \geq s} \).
\end{enumerate}
\end{remark}

\bigskip
\noindent{\bf Acknowledgements
}Sun and Wang were partially supported by NSFC grant No. 12271089, and C. Zhang was partially supported by NSFC grant No. 12101111. We would like to thank the referees for their valuable feedback, which has allowed us to correct several typographical errors and improve the paper's overall presentation.

\bibliographystyle{plain}

\end{document}